\documentclass[psamsfonts]{amsart}

\usepackage{amssymb,amsfonts}
\usepackage[all,arc]{xy}
\usepackage{enumerate}
\usepackage{mathrsfs}

\newtheorem{thm}{Theorem}[section]
\newtheorem{cor}[thm]{Corollary}
\newtheorem{prop}[thm]{Proposition}
\newtheorem{lem}[thm]{Lemma}

\theoremstyle{definition}
\newtheorem{defn}[thm]{Definition}

\theoremstyle{remark}

\makeatletter
\let\c@equation\c@thm
\makeatother
\numberwithin{equation}{section}

\title{Stable Type of the Mapping Class Group}

\author{Ilya Gekhtman}

\date{13 October 2013}

\begin{document}

\begin{abstract} 
We use dynamics of the Teichm$\ddot{\mathrm{u}}$ller geodesic flow to  show that the action of the mapping class group on the space of projective measured foliations has stable type $III_{\lambda}$ for some $\lambda>0$.
We do this by generalizing a criterion due to Bowen for a number to be in the stable ratio set, and proving some Patterson-Sullivan type results for the Thurston measure on $PMF$.
\end{abstract}

\maketitle

\tableofcontents

\section{Introduction and Statement of Results}
Let $(X,d,\nu)$ be a compact metric space endowed with a probability measure and $G$ a countable group acting quasi-invariantly on $(X,\nu)$.

The ratio set of the action, denoted by $RS(G\curvearrowright (X,\nu))$ is the essential range of the Radon-Nikodym cocycle.

\begin{defn}[Ratio Set]
A number $r\in \mathbb{R}$ is said to be in $RS(G\curvearrowright (X,\nu))$ if for every positive measure set $A\subset X$ and $\epsilon>0$ there is a subset $A'\subset A$ of positive measure and an nonidentity element $g\in \Gamma$ such that 
\begin{itemize}
\item $gA'\subset A$
\item $|\frac{d\nu\circ g}{d\nu}(b)-r|\leq \epsilon$ for all $b\in A'$.
\end{itemize}
The extended real number $+\infty$ is said to be in $RS(G\curvearrowright (X,\nu))$ if and only if for every positive measure
set $A\subset X$ and $n>0$ there exists a positive measure subset $A'\subset A$ and an element $g\in G$ such that 

\begin{itemize}
\item $gA'\subset A$
\item $\frac{d\nu\circ g}{d\nu}(b)>n$ for all $b\in A'$.
\end{itemize}
\end{defn}

In \cite{BN11}, Bowen and Nevo defined the stable ratio set  $SRS(G\curvearrowright (X,\nu))$ to be intersection over all probability measure preserving actions $G\curvearrowright (Y,\kappa)$ of the ratio sets of the product actions $G\curvearrowright (X\times Y, \nu\times \kappa)$.

\begin{defn}[Stable Ratio Set]
A number $r\in \mathbb{R}\cup\{\infty\}$ is in the stable ratio set $SRS(G\curvearrowright (X,\nu))$ if $r\in RS(G\curvearrowright (X\times Y,\nu\times \kappa))$ for every probability measure preserving action $G\curvearrowright (Y,\kappa)$ 
\end{defn}

By \cite{FM77} if the action $G\curvearrowright (X,\nu)$ is ergodic and nonatomic, $RS(G\curvearrowright (X,\nu))\setminus \{0,\infty\}$ is a closed multiplicative subgroup of $\mathbb{R}$ and can thus be classified as one of the following types:
\begin{itemize}
\item $II$ if  $RS(G\curvearrowright (X,\nu))=\{1\}$
\item $III_{0}$ if $RS(G\curvearrowright (X,\nu))=\{0, 1, \infty\}$
\item $III_{\lambda}$, $\lambda>1$ if $RS(G\curvearrowright (X,\nu))=\{0, \lambda^{n},\infty: n\in \mathbb{Z}\}$
\item $III_{1}$ if  $RS(G\curvearrowright (X,\nu))=[0,\infty]$
\end{itemize}

The action of $G$ on $(X,\nu)$ is called weak mixing if for every probability measure preserving ergodic action of $G$ on a space $(K,\mu)$ the induced action of $G$ on $(X\times K, \nu\times \mu)$ is ergodic.

It follows that if $G\curvearrowright (X,\nu)$ is weak mixing, its stable ratio set  is one of the four types just described. This is called the stable type of the action.

In \cite{BN11},  Bowen and Nevo used this notion to prove pointwise ergodic theorems for a large class of (nonamenable) groups, with the principal condition being that they admit a nonsingular action of stable type $III_{\lambda}$ for some $\lambda>0$.

In \cite{B12} Bowen proves that for $G$ a Gromov hyperbolic group, $X$ its Gromov boundary, and $\nu$ the Patterson-Sullivan measure on $X$, if $G\curvearrowright (X,\nu)$ is weak mixing then it has stable type $III_{\lambda}$ for some $\lambda\in (0,1]$.
In this paper, we prove an analogous result for the mapping class $Mod(S)$ of a surface $S$ of genus at least 2 acting on the space $PMF$ of projective measured foliations with the Thurston measure. 
\begin{thm}
The action $Mod(S)\curvearrowright (PMF,\nu)$ has stable type $III_{\lambda}$ for some $\lambda>0$.
\end{thm}
We prove  Theorem 1.3 by introducing the notion of a family of functions $\Upsilon_{n}:G\times X\times X, n\in \mathbb{N}$ admissible relative to a collection of subsets $\Omega(n,m),n,m\in \mathbb{N}$ of $PMF$ for $G\curvearrowright (X,\nu)$. This generalizes Bowen's notion of admissible family from \cite{B12}.
We show in Section 2  that the existence of a relatively admissible family for a weakly mixing action $G\curvearrowright (X,\nu)$ implies that action has stable type $III_{\lambda}$ for some $\lambda>0$. We then show in Sections 4 and 5 that there exists a relatively admissible family for the  action of the mapping class group $Mod(S)$ on $PMF$ with the Thurston measure. 

 While the Teichm$\ddot{\mathrm{u}}$ller space $Teich(S)$ is not globally hyperbolic in any reasonable sense (eg it is not Gromov hyperbolic and not $CAT(0)$), some parts of it exhibit many aspects of hyperbolicity. In particular, Teichm$\ddot{\mathrm{u}}$ller geodesic segments spending a uniform proportion of the time over compact parts of moduli space resemble those in Gromov hyperbolic spaces.   The Thurston measure can be considered as a conformal density for the Teichm$\ddot{\mathrm{u}}$ller metric, and in Section 4 we use this conformal property to prove a relative analogue of Sullivan's shadow lemma estimating shadows from a fixed origin of balls in $Teich(S)$ where the connecting segment spends a  uniform proportion in the thick part. The general strategy of the proof is to use reccurence estimates of the Teichm$\ddot{\mathrm{u}}$ller geodesic flow to show that various quantities are asymptotically dominated by the contribution of the thick part.  This allows us to construct "relative" versions of Bowen's admissible families.

Roughly,  the subsets $\Omega(n,m)$ consist of elements of $PMF$ corresponding to geodesic rays from the basepoint $o$ that look hyperbolic near distance $n$ from $o$, with the hyperbolicity weakening as $m$ grows.
The functions $\Upsilon_{n}$ are roughly defined as follows.

$$\Upsilon_{n}(g,b,b')=\frac{1_{Y_{n}(b)}(g)}{|Y_{n}(b)|}\frac{1_{Z_{n}(g)}(b')}{\nu(Z_{n}(g))}.$$
Here, $|Y_{n}(b)|$ denotes the cardinality of $Y_{n}(b)$.
A mapping class element $g$ is in $Y_{n}(b)$ if it moves $o$ a distance of approximately $2n$, $[o,go]$ fellow travels $[o,b)$ for time slightly more than halfway and $[o,go]$ keeps exhibiting hyperbolic behavior after separating from $[o,b)$;
$Z_{n}(g)$ is the subset of $PMF$ consisting of those $b'$ such that $[o,go]$ follows $[o,b')$ slightly less than half way and $[o,b')$ keeps exhibiting hyperbolic behavior after separating from $[o,go]$. 

The connection with stable type is made by the following:

\begin{thm}
Suppose $\Upsilon_{n}, n\in \mathbb{N}$ is admissible relative to $\Omega(m,n),m,n\in \mathbb{N}$ for $G\curvearrowright (X,\nu)$
For each $m$ let $\zeta_{m}$ be any weak-* limit of the $\zeta_{n,m}$ as $n\to \infty$. Let $\zeta$ be any weak-* limit of the $\zeta_{m}$. Then $e^{T}$ is contained in the stable ratio set of $G\curvearrowright (X,\nu)$ for every $T$ in the support of $\zeta$.
\end{thm}

It seems that a simplified version of our argument in Section 5 can be used to construct pseudo-admissible families for the actions of nonuniform lattices in manifolds of pinched variable negative curvature on their boundary spheres, proving an analogue of Theorem 1.3 for these actions.
\subsection{Acknowledgements}
I would like to thank my advisor, Alex Eskin, for his guidance, encouragement and useful conversations. I would like to thank Lewis Bowen for explaining the Bowen-Nevo notion of stable type to me at UCLA in April 2013, and IPAM for making my visit there possible.  I would like to thank Jayadev Athreya, Moon Duchin and Howard Masur for useful conversations.

\section{Relatively-Admissible Families}
Let $(X,d,\nu)$ be a compact metric space endowed with a probability measure and $G$ act ergodically and quasi-invariantly on $X$.
The action of $G$ on $(X,\nu)$ is called weak mixing if for every probability measure preserving ergodic action of $G$ on a space $(K,\mu)$ the induced action of $G$ on $(X\times K, \nu\times \mu)$ is ergodic.
\begin{defn}[Relatively Admissible Families]
A  family of functions $\Upsilon_{n}: G\times X\times X\to \mathbb{R}$, $n,m\in \mathbb{N}$ will be called admissible relative to a family of closed subsets $\Omega(n,m)\subset X$ if:
\begin{itemize}

\item There are $D(m)>0$ with $\lim_{m\to \infty}D(m)=0$ such that $$\nu(\Omega(n,m))>1-D(m)$$ for all $n,m\in \mathbb{N}$

\item For each $m$ there is a function $f_{m}:\mathbb{N}\to \mathbb{R}$ with $f_{m}(n)\to 0$ as $n\to \infty$ such that for all $(g,b,b')$ with $b\in \Omega(n,m)$ and $\Upsilon_{n}(g,b,b')>0$ we have $d(b,b')<f_{m}(n)$ and $d(g^{-1}b,g^{-1}b')<f_{m}(n)$

\item Let $$R(g,\eta)=\log \frac{d{\nu\circ g}}{d\nu}(\eta).$$ For each $m$ there are constants $C(m)>0,N(m)>0$ such that if $n>N(m)$ then  

$$|R(g^{-1},b)|+|R(g^{-1},b')|<C(m)$$ and 

$$|R(g^{-1},b)-R(g^{-1},b')|\geq 1/C(m)$$ for a.e. $(g,b,b')$ with $b\in \Omega(n,m)$ and $\Upsilon_{n}(g,b,b')>0$.

\item $\sum_{g \in G}\int \Upsilon_{n}(g,b,b')d\nu(b')=1$ for every $n>0$ and $b \in PMF$.

\item There exists constants $C(m)>0$ such that the following three quantities are bounded above by $C(m)$ for all $n>N(m)$
$$\int_{b\in \Omega(n,m)} \sum_{g \in G} \Upsilon_{n}(g,b,b') d\nu(b)$$ for a.e. $b'\in X$
\ \
$$\int_{b\in \Omega(n,m)} \sum_{g \in G} \Upsilon_{n}(g,b,gb') \frac{d{\nu\circ g}}{d\nu}(b') d\nu(b)$$ for a.e. $b'\in X$
\ \
$$\int \sum_{g \in G} 1_{\Omega(n,m)}(gb)\Upsilon_{n}(g,gb,b') \frac{d{\nu\circ g}}{d\nu}(b') d\nu(b')$$ for a.e. $b\in X$
\end{itemize}
\end{defn}

Define a measure $\zeta_{n,m}$ on $\mathbb{R}$ by
$$\zeta_{n,m}(E)=\sum_{g\in G} \int \int 1_{E}(R(g^{-1},b')-R(g^{-1},b)) 1_{\Omega(m,n)}(b)\Upsilon_{n}(g,b,b')d\nu(b) d\nu(b')$$
In this section we will prove:
\begin{thm}
Suppose $\Upsilon_{n}, n\in \mathbb{N}$ is admissible relatively to $\Omega(m,n),m,n\in \mathbb{N}$ for $G\curvearrowright (X,\nu)$
For each $m$ let $\zeta_{m}$ be any weak-* limit of the $\zeta_{n,m}$ as $n\to \infty$. Let $\zeta$ be any weak-* limit of the $\zeta_{m}$. Then $e^{T}$ is contained in the stable ratio set of $G\curvearrowright (X,\nu)$ for every $T$ in the support of $\zeta$.
\end{thm}

Note, by the third bullet of Definition 2.1 and the fact that the $\zeta_{n,m}$ are measures of total mass $1-D(m)<||\zeta_{n,m}||<1$, for each $m$ such a weak-* limit $\zeta_{m}$ must exist and have support bounded away from $0$. Moreover since $\zeta_{n,m}(E)\geq \zeta_{n,m'}(E)$ for $m>m'$ and all measurable $E$ we have $\zeta_{m}(E)\geq \zeta_{m'}(E)$ so any weak * limit $\zeta$ of the $\zeta_{m}$ is a probability measure whose support has a nonzero point.

It follows that the stable ratio set is not contained in $\{0, 1, \infty\}$.

We thus obtain 
\begin{cor}
If  $G\curvearrowright (X,\nu)$ is weak mixing, and there exists a relatively admissible family for this action, then the action has stable type $III_{\lambda}$ for some $\lambda \geq 1$.
\end{cor}

Define the following operators.

$$L_{n,m}f(b,t)=\sum_{g\in G}\int f(b,t+R(g^{-1},b')-R(g^{-1},b))1_{\Omega(n,m)}(b)\Upsilon_{n}(g,b,b')d\nu(b')$$
$$W_{n,m}f(b,t)=\sum_{g\in G}\int f(b',t)1_{\Omega(n,m)}(b)\Upsilon_{n}(g,b,b')d\nu(b')$$
$$X_{n,m}f(b,t)=\sum_{g\in G}\int f(g^{-1}b',t+R(g^{-1},b'))1_{\Omega(n,m)}(b)\Upsilon_{n}(g,b,b')d\nu(b')$$
$$Y_{n,m}f(b,t)=\sum_{g\in G}\int f(g^{-1}b,t+R(g^{-1},b'))1_{\Omega(n,m)}(b)\Upsilon_{n}(g,b,b')d\nu(b')$$

Let $A_{t}(r):\mathbb{R}\to \mathbb{R}$ be addition by $t$ so that $A_{t}(r)=t+r$.
Let $\theta$ be a probability measure on $\mathbb{R}$ equivalent to Lebesgue measure such that for every $D>0$ there exists some $D'>0$ such that  for every $t_0 \in \mathbb{R}$ with $|t_0|\leq D$ we have 
$$\frac{d \theta \circ A_{t_0}}{d\theta} \leq D'$$
For example, we could choose $\theta$ to satisfy
$$d\theta=(1/2)e^{-|t|}dt$$
\begin{lem}
For each $m$ there exists a $C_{1}(m)>0$ independent of $n$  such that the $L^{1}$ norm of each $W_{n,m},X_{n,m},Y_{n,m}$ is bounded above by $C_{1}(m)$.
\end{lem}

\begin{proof}
Let $f\in L^{1}(\nu\times\theta)$ be nonnegative.\\

{\bf Case $W_{n,m}$}

Because $\sum_{g\in \Gamma} \int_{\Omega(n,m)}  \Upsilon_{n}(g,b,b')~d\nu(b) \le C(m)$,

$$||W_{n,m} f|| =   \int \int |W_{n,m} f| d\nu d\theta =  \sum_{g\in \Gamma}\int \int \int f(b',t) 1_{\Omega(n,m)}(b)\Upsilon_{n}(g,b,b')  d\nu(b')  d\nu(b)d\theta(t)$$
$$\leq C(m) \int \int f(b',t)  d\nu(b')d\theta(t) = C(m) ||f||$$

{\bf Case $X_{n,m}$}

Because $\sum_{g\in \Gamma} \int_{\Omega(n,m)}  \Upsilon_{n}(g,b,gb')R(g,b')~d\nu(b) \le C(m)$,
$$||X_{n,m} f|| =   \int \int |X_{n,m} f| d\nu d\theta =  \sum_{g\in \Gamma}\int \int \int f(g^{-1}b',t+R(g^{-1},b))1_{\Omega(n,m)}(b) \Upsilon_{n}(g,b,b')  d\nu(b')  d\nu(b)d\theta(t)$$
$$= \sum_{g\in \Gamma}\int \int \int f(g^{-1}b',t)\frac{d \theta \circ A_{-R(g^{-1},b')}}{d\theta}(t)1_{\Omega(n,m)}(b) \Upsilon_{n}(g,b,b')  d\nu(b')  d\nu(b)d\theta(t)$$ 
$$\leq C'(m) \sum_{g\in \Gamma}\int \int \int f(g^{-1}b',t)1_{\Omega(n,m)}(b) \Upsilon_{n}(g,b,b')  d\nu(b')  d\nu(b)d\theta(t)$$
$$\leq C'(m) \sum_{g\in \Gamma}\int \int \int f(b',t) \frac{d\nu\circ g}{d\nu}(b')1_{\Omega(n,m)}(gb)\Upsilon_{n}(g,b,gb')  d\nu(b')  d\nu(b)d\theta(t)$$
$$\leq C'(m)C(m)\int \int f(b',t)d\nu(b')d\theta(t)=C'(m)C(m)||f||$$

{\bf Case $Y_{n,m}$}
Because $\sum_{g\in \Gamma} \int 1_{\Omega(n,m)}(gb) \Upsilon_{n}(g,gb,b')R(g,b')~d\nu(b) \le C(m)$,

$$||Y_{n,m} f|| =   \int \int |Y_{n,m} f| d\nu d\theta = \sum_{g\in \Gamma}\int \int \int f(g^{-1}b,t+R(g^{-1},b'))1_{\Omega(n,m)}(b) \Upsilon_{n}(g,b,b')  d\nu(b')  d\nu(b)d\theta(t)$$
$$= \sum_{g\in \Gamma}\int \int \int f(g^{-1}b,t)\frac{d \theta \circ A_{-R(g^{-1},b')}}{d\theta}(t)1_{\Omega(n,m)}(b) \Upsilon_{n}(g,b,b')  d\nu(b')  d\nu(b)d\theta(t)$$ 
$$\leq C'(m) \sum_{g\in \Gamma}\int \int \int f(g^{-1}b,t) 1_{\Omega(n,m)}(b)\Upsilon_{n}(g,b,b')  d\nu(b')  d\nu(b)d\theta(t)$$
$$\leq C'(m) \sum_{g\in \Gamma}\int \int \int f(b,t) \frac{d\nu\circ g}{d\nu}(b) 1_{\Omega(n,m)}(gb)\Upsilon_{n}(g,gb,b')  d\nu(b')  d\nu(b)d\theta(t)$$
$$\leq C'(m)C(m)\int \int f(b,t)d\nu(b')d\theta(t)=C'(m)C(m)||f||$$
\end{proof}

\begin{lem}
For all  $f\in L^{1}(\nu \times \theta)$,
$$\lim \sup_{n\to \infty}||W_{n,m}f-f||_{1}\leq D(m)||f||_{1}$$
$$\lim_{n\to \infty}||X_{n,m}f-Y_{n,m}f||_{1}=0$$
\end{lem}

\begin{proof}
Without loss of generality let $f$ be a continuous function with compact support on $X\times \mathbb{R}$.
Let  $Var_{n,m}(f)=\sup_{d(x,y)<f_{m}(n)}|f(x,t)-f(y,t)|$.
Note $Var_{n,m}(f)\to 0$ as $n\to \infty$.
If $b\notin \Omega(n,m)$ then clearly $W_{n,m}f=0$. On the other hand, since  $\sum_{g \in G}\int \Upsilon_{n,m}(g,b,b')d\nu(b')=1$ for every $n,m>0$ and $b \in PMF$ and  for every $(g,b,b')$ with $b\in \Omega(n,m)$ and $\Upsilon_{n,m}(g,b,b')>0$ we have $d(b,b')<f_{m}(n)$ and $d(g^{-1}b,g^{-1}b')<f_{m}(n)$ we have
$$|W_{n,m}f(b,t)-f(b,t)|\leq Var_{n,m}(f)$$ whenever $b\in \Omega(n,m)$.
Thus $$||W_{n,m}f-f||_{1}\leq Var_{n,m}(f)+D(m)||f||_{\infty}$$ and hence
$$\lim \sup_{n\to \infty}||W_{n,m}f-f||_{1}\leq D(m)||f||_{\infty}\leq D(m)||f||_{1}$$
 Since compactly supported continuous functions are $L^{1}$ dense and the $||W_{n,m}||_{1}\leq C(m)$ for all $n$, the first statement of the lemma follows. The second statement is proved similarly.
\end{proof}
Let $G$ act on $X\times \mathbb{R}$ by $g(b,t)=(gb, t+R(g,b))$ and on  $L^{1}(\nu \times \theta)$ by $g\dot f=f\circ g^{-1}$
\begin{prop}

For every $G$ invariant function $f\in L^{1}(\nu \times \theta)$ we have

$$\lim_{n\to \infty}||f-L_{n,m}f||\leq D(m)||f||$$
\end{prop}
\begin{proof}
Since $f$ is $G$ invariant we have $X_{n,m}=W_{n,m}$ and $Y_{n,m}=L_{n,m}$.
Now $$||X_{n,m}f-Y_{n,m}f||=||W_{n,m}f-L_{n,m}f||\geq ||f-L_{n,m}f||-||f-W_{n,m}f||$$
As $n\to \infty$ we have $$||X_{n,m}f-Y_{n,m}f||\to 0$$ and $$\lim \sup_{n\to \infty}||f-W_{n,m}f||\leq D(m)||f||$$ proving the proposition.
\end{proof}
This immediately implies
\begin{cor}
$$\lim_{m\to \infty}\lim \sup_{n\to \infty}||f-L_{n,m}||=0$$
\end{cor}

Recall the measure $\zeta_{n,m}$ on $\mathbb{R}$ defined by
$$\zeta_{n,m}(E)=\sum_{g\in G} \int \int 1_{E}(R(g^{-1},b')-R(g^{-1},b)) 1_{\Omega(n,m)}(b)\Upsilon_{n,m}(g,b,b')d\nu(b) d\nu(b')$$
The following is Theorem 2.2  with "ratio set" in place of "stable ratio set".
\begin{prop}
For each $m$ let $\zeta_{m}$ be any weak-* limit of the $\zeta_{n,m}$ as $n\to \infty$. Let $\zeta$ be any weak-* limit of the $\zeta_{m}$. Then $e^{T}$ is contained in the ratio set of $G\curvearrowright (X,\nu)$ for every $T$ in the support of $\zeta$.
\end{prop}

\begin{lem}[Lemma 3.8 in \cite{B12}]
Suppose $e^{T}$ is not in the ratio set of $G$ acting on $(B,\nu)$. Then there exists an $\epsilon > 0$ and a $G$-invariant, positive measure set $A\subset B\times \mathbb{R}$ such that for every $(b,t)\in A$ and $t'\in (-\epsilon,\epsilon)$, $(b,t+T+t')\notin A$.
\end{lem}

\begin{proof}[Proof of Proposition 2.8]
Let $T$ be an element of the support of $\zeta$. To obtain a contradiction, suppose that $e^{T}$ it is not in the ratio set. Let $A$ and $\epsilon$ be as in the previous lemma. Let $f$ be the characteristic function of $A$.
Note,
$$L_{m,n}f(b,t)=\int 1_{\Omega(n,m)}(b) f(b,t+t')d\zeta_{n,m,b}(t')$$ where $\zeta_{n,m,b}$ is the probability measure given by $$\zeta_{n,m,b}(E)=\sum_{g\in G} \int 1_{E}(R(g^{-1},b')-R(g^{-1},b))\Upsilon_{n}(g,b,b')d\nu(b').$$
Thus  $$||f-L_{n,m}f||=\int 1_{\Omega(n,m)}(b) |f(b,t)-\int f(b,t+t')d\zeta_{n,m,b}(t')|d\nu(b)d\theta(t)$$ $$\geq \int_{A}1_{\Omega(n,m)}(b) |f(b,t)-\int f(b,t+t')d\zeta_{n,m,b}(t')|d\nu(b)d\theta(t)$$ $$\geq \int_{A\cap (\Omega(n,m)\times \mathbb{R})}\zeta_{m,n,b}((T-\epsilon,T+\epsilon))d\nu(b)d\theta(t)$$ $$= (\nu\times \theta)(A\cap (\Omega(n,m)\times \mathbb{R}))\zeta_{m,n}(T-\epsilon,T+\epsilon)$$ $$\geq ((\nu\times \theta)(A)-D(m))\zeta_{m,n}(T-\epsilon,T+\epsilon)$$
The second inequality holds because by the Lemma 2.9, if $(b,t)\in A$ and $t'\in (T-\epsilon, T+\epsilon)$ then $(b,t+t')\notin A$ so $f(b,t)-f(b,t')=1$.
\\
Fixing $m$ and taking limits as $n\to \infty$ gives
$$D(m) \geq ((\nu\times \theta)(A)-D(m))\zeta_{m}(T-\epsilon,T+\epsilon)$$ and taking the limit as $m\to \infty$ we get
$$(\nu\times \theta)(A) \zeta(T-\epsilon,T+\epsilon)=0$$ contradicting that $T$ is in the support of $\zeta$.
\end{proof}
Thus we obtain that the action of $G$ on $(X,\nu)$ does not have type $III_{0}$ proving  Theorem 1.4 with "ratio set" in place of "stable ratio set".

To prove Theorem 2.2, we will show that given any ergodic  measure preserving action of $G$ on a probability space $(K,\kappa)$ there exists  a topological model for this action and an pseudo-admissible family $\Upsilon'_{n,m}$ for this action with limit measure $\zeta'$ such that if $T$ is in the support of $\zeta$ then $T$ is also in the support of $\zeta'$.
\begin{lem}[Prop 3.10 in \cite{B12}]
Let $\Gamma \curvearrowright (X,\mu)$ be an ergodic pmp action. Then there exists a compact metric space $(K,d_K)$ with a Borel probability measure $\kappa$ and a continuous action $\Gamma \curvearrowright K$ such that
\begin{itemize}
\item $\Gamma \curvearrowright (X,\mu)$ is measurably conjugate to $\Gamma \curvearrowright (K,\kappa)$
\item for every $\epsilon>0$ and $x,y \in K$,
$$1/3 \leq \frac{\kappa(B(x,\epsilon))}{\kappa(B(y,\epsilon))} \leq 3$$
where for example, $B(x,\epsilon)=\{z\in K:~d_K(x,z)\leq \epsilon\}$.
\end{itemize}
\end{lem}
\begin{proof}[Proof of Theorem 2.2]
 Let $\Gamma \curvearrowright (K,\kappa)$ be an ergodic probability measure preserving action. By Lemma 2.10, we may assume that $(K,d_K)$ is a compact metric space such that for every $\epsilon>0$ and $x,y \in K$,
$$1/3 \leq \frac{\kappa(B(x,\epsilon))}{\kappa(B(y,\epsilon))} \leq 3.$$
Given an integer $n\ge 1$ and $g\in \Gamma$, let $0<\rho(n,g)<1/n$ be such that for every $x,y \in K$ with $d_K(x,y)\leq \rho(n,g)$, $d_K(g^{-1}x,g^{-1}y) \leq 1/n$. 
Define $\Upsilon'_{n}:\Gamma \times X\times K \times X \times K \to \mathbb{R}$ by 
$$\Upsilon'_{n}(g,b,k,b',k') := \frac{1_{B( k,\rho(n,g))}(k')\Upsilon_{n}(g,b,b')}{ \kappa(B(k,\rho(n,g))) }.$$
It is an easy exercise using the above estimates to check that $\{\Upsilon'_{n}\}_{n=1}^\infty$, $\Omega(m,n)\times K$ is an admissible family for $G \curvearrowright (B\times K,\nu \times \kappa)$ with $d_{B\times K}$, a metric on $B\times K$, given by $d_{B\times K}((b,k),(b',k'))=d_B(b,b')+ d_K(k,k')$.
Since $G\curvearrowright (K,\kappa)$ is measure preserving,
 $$R(g,b,k):= \log \frac{d(\nu\times \kappa)\circ g}{d(\nu\times \kappa)}(b,k) = R(g,b).$$
Thus, for any $E\subset \mathbb{R}$
$$\zeta_{n,m}(E) =\sum_{g \in \Gamma} \int \int   1_E\left( R(g^{-1},b') - R(g^{-1},b) \right) 1_{\Omega(n,m)}(b) \Upsilon_{n}(g,b,b') ~d\nu(b')d\nu(b)$$
$$= \sum_{g \in \Gamma} \int \int   1_E\left( R(g^{-1},b',k') - R(g^{-1},b,k) \right) 1_{\Omega(n,m)}(b) \Upsilon'_{n}(g,b,k,b',k') ~d\nu\times \kappa(b',k')d\nu \times \kappa(b,k)$$
Thus, Prop 3.8 implies  the ratio set of the action $\Gamma \curvearrowright (B \times K,\nu\times \kappa)$ contains $e^T$.  Since $\Gamma  \curvearrowright (K,\kappa)$ is arbitrary, the proof is complete.
\end{proof}

\section{Background on the Geometry of  Teichm$\ddot{\mathrm{u}}$ller Space}
Let $S$ be a closed surface of genus $g\geq 2$. Let $Mod(S)$ be the associated mapping class group.  
The  Teichm$\ddot{\mathrm{u}}$ller space $Teich(S)$ is the space of all marked  or hyperbolic structures on $S$ up to isotopy.
We endow it with the Teichm$\ddot{\mathrm{u}}$ller metric.
Thurston showed that $Teich(S)\cong \mathbb{R}^{6g-6}$ has a natural compactification by the space $PMF\cong S^{6g-7}$ of projective classes of measured foliations $MF$ on $S$, which has many analogies with the compactification of hyperbolic space by its boundary sphere \cite{FLP}.
 The space $Q(S)$ of quadratic differentials can be thought as a cotangent bundle of $Teich(S)$. A quadratic differential $q$ is determined by its vertical and horizontal measured foliations $q^{+}$ and $q^{-}$ respectively.
 For each $o\in Teich(S)$ $\eta$ $PMF$ there is a unique Teichm$\ddot{\mathrm{u}}$ller geodesic through $o$ in the direction of $\eta$. Moreover, if $\eta\in PMF$ is uniquely ergodic, for any $\eta'\in PMF$ there is a unique Teichm$\ddot{\mathrm{u}}$ller geodesic with forward and backward directions $\eta$ and $\eta'$ \cite{HuMa}.
By  Masur' criterion for unique ergodicity, \cite{Mas1} geodesics in non-uniquely ergodic directions eventually exit forever every thick part $Teich{\epsilon}(S)$.
Furthermore, if $q^{+}$ is uniquely ergodic, the geodesic ray $g_{t}q$ converges to $[q^{+}]$ \cite{Mas2}.
The Busemann cocycle $$\beta_{z}(x,y)=d(x,z)-d(y,z),$$ $x,y,z\in Teich(S)$ extends continuously to uniquely ergodic $z\in PMF$.
There is a unique probability measure $\mu$ of maximal entropy for the Teichm$\ddot{\mathrm{u}}$ller geodesic flow on $Q^{1}(S)/Mod(S)$, the so called Masur-Veech measure, and it is in the Lebesgue measure class with respect the period coordinates on $Q(S)$.
Its entropy is $h=6g-6$.
Let $m$ be the Thurston measure on $MF$. The measured foliations which are not uniquely ergodic have $m$ measure $0$ \cite{KMS}.
For each $x\in Teich(S)$ define
$$\nu_{x}(A)=m(\{\eta \in MF: [\eta] \in A, Ext_{x}\eta\leq 1\})$$ 
We call these normalized Thurston measures on $PMF$.
The measures $\nu_{x},x\in Teich(S)$ form a conformal density for the action of $Mod(S)$ on $PMF$ in the sense that 
$$\nu_{x}\circ g^{-1}=\nu_{gx}$$ and
$$\frac{d\nu_{x}}{d\nu_{y}}(\eta)=e^{h\beta_{\eta}(x,y)}$$
for all $g\in Mod(S), x,y\in Teich(S)$ and $\eta \in PMF$ uniquely ergodic.
We can write the lift $\tilde{mu}$ of $\mu$ to $Q^{1}(S)$ as
$$d\tilde{\mu}(q)=\exp(h\beta_{[q^{+}]}(o,\pi(q)))\exp(h\beta_{[q^{-}]}(o,\pi(q)))   d\nu_{o}([q^{+}])d\nu_{o}([q^{-}])$$ for any $o\in Teich(S)$. The expression makes sense because almost every quadratic differential has uniquely ergodic vertical and horizontal measured foliations.
The measures $\mu$ and $\nu_{x}$ are thus the analogues in the Teichm$\ddot{\mathrm{u}}$ller setting of Bowen-Margulis and Patterson-Sullivan measures respectively.

For $\epsilon>0$ let $Teich_{\epsilon}(S)$ be the $\epsilon$-thick part of $Teich(S)$, which consists of all hyperbolic structures on $S$ with no nontrivial curves of hyperbolic length less than $\epsilon$.
By Mumford's criterion $M_{\epsilon}(S)=Teich_{\epsilon}(S)/Mod(S)$ is compact for all $\epsilon>0$.
The following is Theorem A of \cite{DDM} due to Dowdall-Duchin-Masur
\begin{prop}
For each $\epsilon,\theta>\theta'>0$ there is an $L>0$ and $\delta>0$ such that if $I\subset [x,y]\subset Teich(S)$ is a geodesic subinterval of length at least $L$ and a proportion of at least $\theta$ of $I$ lies in $Teich_{\epsilon}(S)$, then for all $z\in Teich(S)$ the intersection $I \cap Nbhd_{\delta}([x, z] \cup [y, z])$ has measure at least $\theta' l(I)$. 
\end{prop}
The following property of Teichm$\ddot{\mathrm{u}}$ller geodesics, also indicative of hyperbolicity in the thick part, is due to Rafi \cite{Ra}.

\begin{prop}
For each $A>0$ and $\epsilon>0$ there exists a constant $K>0$ such that for points $x, x', y, y'\in Teich_{\epsilon}(S)$ with $d_{T}(x,x')\leq A$ and $d_{T}(y,y')\leq A$ the geodesic segments $[x,y]$ and $[x',y']$ $K$-fellow travel in a parametrized fashion, and for $\eta\in PMF$ such that $[x,\eta)$ and $[x',\eta)$ are contained in $Teich_{\epsilon}(S)$, the geodesic rays $[x,\eta)$ and $[x',\eta)$ $K$-fellow travel in a parametrized fashion.
\end{prop}

\section{Construction of a Pseudo-Admissible Family for $Mod(S)\curvearrowright PMF$}
Let $\epsilon>0$ be such that $\nu(M_{\epsilon}(S))>0.9999$.
\\
Let $L$ and $\delta$ be the ones provided by Proposition 3.1 for this $\epsilon$ and $\theta=0.9$, $\theta'=0.8$
\\
Let $K$ be the one provided by Proposition 3.2 with $2\delta$ in place of $A$.
\\
Let $0<\epsilon'<\epsilon$ be such that $Nbhd_{5K}Teich_{\epsilon}(S)\subset Teich_{\epsilon'}(S)$. 
\\
Let $L_{1}>L$ and $\delta_{1}>\delta$ be the ones provided by Proposition 3.1 for  $\epsilon'$ in place of $\epsilon$ and $\theta=0.6$, $\theta'=0.55$
\\
Let $K_{1}>K$ be the one provided by Proposition 3.2 with $2\delta_{1}$ in place of $A$.
\\
Let $0<\epsilon''<\epsilon'$ be such that $Nbhd_{5K_{1}}Teich_{\epsilon'}(S)\subset Teich_{\epsilon''}(S)$. 
\\
Let  $L_{2}>L_{1}$ and $\delta_{2}>\delta_{1}$ be the ones provided by Proposition 3.1 for  $\epsilon''$ in place of $\epsilon$ and $\theta=0.6$, $\theta'=0.55$
\\
Let $K_{2}>K_{1}$ be the one provided by Proposition 3.2 with $2\delta_{2}$ in place of $A$.
Assume without loss of generality that $\delta$ is more than twice the diameter of a fundamental domain of $Teich_{\epsilon}(S)$ and  $\delta_{1}$ is more than twice the diameter of a fundamental domain of $Teich_{\epsilon'}(S)$.
\\
Define $\Omega(n,m)\subset PMF$ to be the set of $b \in PMF$ such that  for any $t>m$ at least $0.9999$ of each of $\gamma_{o,b}([n-t,n])$ and $\gamma_{o,b}([n,n+t])$  lies in $Teich_{\epsilon}(S)$.
\\
Note, it follows that if  $n>2000m$ at least $0.9$ of $\gamma_{o,b}([n-i m,n-(i-1)m])$ lies in $Teich_{\epsilon}(S)$ for $i=-1000,...,1000$
\\
For each $b\in \Omega(m,n)$ such that $b\notin \Omega(n,k)$ for $k<m$ define $Y_{n}\subset Mod(S)$ to be the set of $g \in Mod(S)$ such that:
\begin{itemize}
\item $$d(o,go)\in (2n-20m, 2n+20m)$$ 
\item $$-100m\leq \beta_{\eta}(go,o)\leq -50m$$
\item  At least 99 percent of $\gamma_{go,o}[n-121m-t,n-121m]$ lies in $Teich_{\epsilon'}(S)$ for all $n-121m\geq t\geq m$.
\end{itemize}

For each such $g\in Mod(S)$ let $Z_{n}(g)$ be the set of $b'\in PMF$ such that 
\begin{itemize}
\item At least $90$ percent of $b'([n-9m,n-8m]$ lies in $Teich_{\epsilon'}(S)$.
\item For every $t\leq n-20m$, $d(\gamma_{o,go}(t),\gamma_{o,b'}(t))\leq K_{1}$.
\item For some $t\in [n-10m,n-9m]$, $d(\gamma_{o,go}(t),\gamma_{o,b'}(t)\geq K_{1}$.
\end{itemize}

For each $b,b'\in PMF$ and $g\in Mod(S)$ let $$\Upsilon_{n}(g,b,b')=\frac{1_{Y_{n}(b)}(g)}{|Y_{n}(b)|}\frac{1_{Z_{n}(g)}(b')}{\nu(Z_{n}(g))}.$$

Roughly,  the $\Omega(n,m)$ are elements of $PMF$ corresponding to geodesic rays from the basepoint $o$ that look hyperbolic near distance $n$ from $o$, with the hyperbolicity weakening as $m$ grows;
$g\in Y_{n}(b)$ if it moves $o$ a distance of approximately $2n$, $[o,go]$ fellow travels $[o,b)$ for time slightly more than halfway and $[o,go]$ keeps exhibiting hyperbolic behavior after separating from $[o,b)$;
$b'\in Z_{n}(g)$ if $[o,go]$ follows $[o,b')$ slightly less than half way and $[o,b')$ keeps exhibiting hyperbolic behavior after separating from $[o,go]$.

We will prove:

\begin{thm} 
The $\Upsilon_{n},\Omega(n,m)$ are admissible relative to $\Omega(n,m)$ for $Mod(S)\curvearrowright (PMF,\nu)$.
\end{thm}

The following  propositions will be proved in the next section by modifying techniques from Gromov hyperbolic geometry and Patterson-Sullivan theory. Propositions 4.3 and 4.4 are derived from a Teichmueller analogue of Sullivan's shadow lemma proved in Proposition 5.1 while Propositions 4.2, 4.5, and 4.6 use estimates of Athreya-Bufetov-Eskin-Mirzakhani \cite{ABEM} on the number of lattice points in a ball.

\begin{prop}
For every $b\in \Omega(n,m)$ and $g\in Y_{n}(b)$ we have $|Y_{n}(b)|\simeq_{m} e^{hn}$.
\end{prop}

\begin{prop}
For every $b\in \Omega(n,m)$ and $g\in Y_{n}(b)$ we have $\nu(Z_{n}(g))\simeq_{m}e^{-hn}$.
\end{prop}

\begin{prop}
For all $b'\in PMF$
$$\nu\{b\in \Omega(n,m):b'\in \bigcup_{g\in Y_{n}(b)}Z_{n}(g)\}\lesssim_{m} e^{-hn}$$
\end{prop}

\begin{prop}
For all $b'\in PMF$
the number of $g\in Mod(S)$ with $gb'\in Z_{n}(g)$ and $g\in Y_{n}(b)$ for some $b\in \Omega(n,m)$ has cardinality
$\lesssim_{m}e^{hn}$.
\end{prop}

\begin{prop}
For each $b\in \Omega(n,m)$
$$|\{g\in Mod(S):gb\in \Omega(n,m),g\in Y_{n}(gb)\}|\lesssim_{m}e^{hn}$$
\end{prop}

\begin{prop}
For each $b\in \Omega(n,m)$, $g\in Y_{n}(b)$, $b'\in Z_{n}(g)$ we have 
$$b,b',g^{-1}b,g^{-1}b'\in pr_{o,g^{-1}o}B_{\delta_{1}}(\gamma_{o,g^{-1}o}(t))$$
 for some $t>n-122m$ with $\gamma_{o,g^{-1}o}(t)\in Teich_{\epsilon'}(S)$
\end{prop}

\begin{prop}
For each $b\in \Omega(n,m)$, $g\in Y_{n}(b)$ and $b'\in Z_{n}(g)$ we have
$-6m\leq \beta_{b'}(go,o)\leq 21m$
\end{prop}
We are now ready to verify the conditions of Definition 1.1.
The first bullet point follows since by ergodicity of the Teichm$\ddot{\mathrm{u}}$ller geodesic flow, almost all geodesic rays from $o$ become equidistributed. The second follows by Proposition 4.7. The third follows by Proposition 4.8 and the definition of $Z_{n}(g)$. The fourth is immediate from the definition of $\Upsilon_{n}$.
We now verify the estimates of the fifth bullet point. 
For the first estimate, note:
$$\int_{b\in \Omega(n,m)} \sum_{g \in Mod(S)} \Upsilon_{n}(g,b,b') d\nu(b)=$$                                                                                    
$$\int_{b\in \Omega(n,m)} \frac{1}{|Y_{n} (b)|}  \sum_{g\in Y_{n}(b)} \frac{1_{Z_{n}(g)}(b')} {\nu(Z_{n}(g))}d\nu(b)$$
$$\lesssim_{m} \int_{b\in \Omega(n,m)} \frac{e^{hn}}{|Y_{n}(b)|}\sum_{g\in Y_{n}(b)}1_{Z_{n}(g)}(b')d\nu(b)$$ $$\leq e^{hn}\nu(\{b\in \Omega(n,m): b'\in \cup_{g\in Y_{n}(b)}Z_{n}(g)\})\lesssim_{m} 1$$

For the second estimate of the fifth bullet point note that if $\Upsilon_{n}(g,b,gb')>0$ then 
$$\frac{d\nu\circ g}{d\nu}(b')=e^{-h\beta_{b'}(g^{-1}o,o)}=e^{h\beta_{gb'}(go,o)}\leq e^{15hm}$$ thus 
$$\int_{\Omega(n,m)} \sum_{g \in Mod(S)} \Upsilon_{n}(g,b,gb') \frac{d{\nu\circ g}}{d\nu}(b') d\nu(b)\lesssim_{m}\int_{\Omega(n,m)} \sum_{g \in Mod(S)} \Upsilon_{n}(g,b,gb')=$$
$$\int_{b\in \Omega(n,m)} \frac{1}{|Y_{n} (b)|}  \sum_{g\in Y_{n}(b)} \frac{1_{Z_{n}(g)}(gb')} {\nu(Z_{n}(g))}d\nu(b)\lesssim_{m}$$
$$\int_{b\in \Omega(n,m)} \sum_{g\in Y_{n}(b)} 1_{Z_{n}(g)}(gb')d\nu(b)= \sum_{g\in Mod(S)}\nu\{b\in \Omega(n,m):g\in Y_{n}(b), gb'\in Z_{n}(g)\}\lesssim_{m}1$$

To see the last inequality note  $$\nu\{b\in \Omega(n,m):g\in Y_{n}(b), gb'\in Z_{n}(g)\}\leq \nu\{b\in \Omega(n,m):gb'\in \bigcup_{k\in Y_{n}(b)}Z_{n}(k)\}\lesssim_{m}e^{-hn}$$  for each $b'$ by Proposition 4.4 and the number of nonzero terms in the sum is at most 
$\lesssim_{m} e^{hn}$ by Proposition 4.5.

For the final estimate, note that if $1_{\Omega(n,m)}(gb) \Upsilon_{n}(g,gb,b')>0$ then 
$$\frac{d\nu\circ g}{d\nu}(b)=e^{-h\beta_{b}(g^{-1}o,o)}=e^{h\beta_{gb}(go,o)}\leq e^{100hm}$$ so 
$$\int \sum_{g \in Mod(S)}1_{\Omega(n,m)}(gb) \Upsilon_{n}(g,gb,b') \frac{d{\nu\circ g}}{d\nu}(b) d\nu(b')\lesssim_{m}\int \sum_{g \in Mod(S)}1_{\Omega(n,m)}(gb) \Upsilon_{n}(g,gb,b')d\nu(b')=$$
$$\int   \sum_{g\in Mod(S)}1_{\Omega(n,m)}(gb) \frac{1_{Y_{n}(gb)}(g)}{|Y_{n} (gb)|} \frac{1_{Z_{n}(g)}(b')} {\nu(Z_{n}(g))}d\nu(b')=\sum_{g\in Mod(S)}1_{\Omega(n,m)}(gb) \frac{1_{Y_{n}(gb)}(g)}{|Y_{n} (gb)|}$$

$$\lesssim_{m} e^{-hn}|\{g\in Mod(S):gb\in \Omega(n,m),g\in Y_{n}(gb)\}|\lesssim_{m}1$$ 
This completes the proof.

\section{Proofs of Propositions in Section 4}

We begin by proving the following analogue of Sullivan's shadow Lemma:
\begin{lem}
For each $r>0$, $\theta,\epsilon>0$ and $R>0$ there exists a $C>0$ with the following property: for every $g\in Mod(S)$ such that  any initial length $\geq R$ segment of $[o,g^{-1}o]$ spends a proportion at least $\theta$ in $Teich_{\epsilon}(S)$ we have 
$$C^{-1}e^{-hd(go,o)}\leq \nu_{o}(pr_{o}B_{r}(go))\leq Ce^{-hd(go,o)}$$
\end{lem}
\begin{proof}
Note,
$$\nu_{o}(pr_{o}B_{r}(\gamma o))=\int_{\eta \in pr_{o}B_{r}(\gamma o)} e^{-h \beta_{\eta}(o,\gamma o)}d\nu_{\gamma o}$$
Furthermore by the triangle inequality if $\eta \in pr_{o} B_{r}(\gamma o)$ we have $$d(o,\gamma o)-2r\leq\beta_{\eta}(o,\gamma o)\leq d(o,\gamma o)$$

Thus $$\nu_{o}(pr_{\gamma^{-1}o}B_{r}(o)) e^{-h d(o,\gamma o)} =\nu_{\gamma o}(pr_{o}B_{r}(\gamma o)) e^{-h d(o,\gamma o)} \leq\nu_{o}(pr_{o}B_{r}(\gamma o))$$
$$\leq e^{2hr-h d(o,\gamma o)}\nu_{\gamma o}(pr_{o}B_{r}(\gamma o))\leq e^{2hr-h d(o,\gamma o)}||\nu_{o}||.$$
So  $$\nu_{o}(pr_{\gamma^{-1}o}B_{r}(o)) e^{-h d(o,\gamma o)}  \leq \nu_{o}(pr_{o}B_{r}(\gamma o))\leq e^{2hr-h d(o,\gamma o)}||\nu_{o}||.$$
This gives an upper bound.

For the lower bound, we need to show that 
 $\nu_{o}(pr_{\gamma^{-1}o}B_{r}(o))$ is bounded away from $0$ independent of $\gamma$ as long as any initial length $\geq R$ segment of $[o,\gamma^{-1}o]$ spends a proportion at least $\theta$ in $Teich_{\epsilon}(S)$ for which it would suffice to show that there is a $E>0$ such that for all $y\in Teich(S)$ such that  any initial length $\geq R$ segment of $[o,y]$ spends a proportion at least $\theta$ in $Teich_{\epsilon}(S)$
$\nu_{o}(pr_{y}B_{r}(o))>E$.
Suppose not. Then there is a sequence of such $y_{n}\in Teich(S)$ converging to $\zeta \in PMF$ with $\nu_{o}(pr_{y_{n}}B_{r}(o))\to 0$. By Masur's criterion, $\zeta$ is uniquely ergodic. Thus, $\nu_{o}(pr_{\zeta}(B_{r}(o)))=0$ which is impossible since $\nu_{o}$ has full support on $PMF$ and $pr_{\zeta}(B_{r}(o))$ contains an open set.
\end{proof}
By Mumford's compactness criterion and Proposition 3.2 we obtain the following corollary.
\begin{cor}
For  every $\theta,\epsilon>0$ and $R>0$ and each  $r>0$  larger than twice the diameter of a fundamental domain of $Teich_{\epsilon}(S)$ there exists a $C>0$ with the following property:
for every $x\in Teich_{\epsilon}(S)$ such that  any initial length $\geq R$ segment of $[x,o]$ spends a proportion at least $\theta$ in $Teich_{\epsilon}(S)$ we have

$$C^{-1}e^{-hd(x,o)}\leq \nu_{o}(pr_{o}B_{r}(x))\leq Ce^{-hd(x,o)}$$
\end{cor}
The next lemma says that at least a uniform proportion of shadows of balls consists of directions which reccur uniformly to the thick part.
\begin{lem}
For every $\theta>0$, $R>0$, $\rho>0$, $\epsilon>0$ with $\mu(M_{\epsilon}(S))\leq \rho$, and  $r>0$ there is a $K>0$  such that for each $\epsilon'>0$ with $Nbd_{K}Teich_{\epsilon}(S)\subset Teich_{\epsilon'}(S)$ there are $M>0$ and $C>0$ such that for every $g\in Mod(S)$ such that  any initial length $\geq R$ segment of $[o,\gamma^{-1}o]$ spends a proportion at least $\theta$ in $Teich_{\epsilon}(S)$ the set of $\eta \in pr_{o}B_{r}(go)$ such that $\gamma_{o,\eta}[d(o,go),d(o,go)+t]$ spends at a proportion of at least $\rho$ in $Teich_{\epsilon'}(S)$ for every $t>M$ has measure at least $Ce^{-hd(o,go)}$.
\end{lem}
\begin{proof}
Let $A(o,go,r,M,\epsilon')$ be  the set of $\eta \in pr_{o}B_{r}(go)$ such that $$\gamma_{o,\eta}[d(o,go),d(o,go)+t]$$ spends at a proportion of at least $\rho$ in $Teich_{\epsilon'}(S)$ for every $t>M$
By conformality of the Thurston measure, $$\frac{\nu_{o}(A(o,go,r,M,\epsilon'))}{\nu_{o}( pr_{o}B_{r}(go))}\geq e^{-4hr}\frac{\nu_{o}(g^{-1}A(o,go,r,M,\epsilon')}{\nu_{o}(pr_{g^{-1}o}B_{r}(o))}.$$

Note, $\nu_{o}(pr_{g^{-1}o}B_{r}(o))>c>0$ for a positive number $c>0$ depending only on $\theta,\epsilon'$ so 
$$\frac{\nu_{o}(A(o,go,r,M,\epsilon'))}{\nu_{o}( pr_{o}B_{r}(go))}\geq D\nu_{o}(g^{-1}A(o,go,r,M,\epsilon'))$$ where $D$ depends only on $\epsilon',r,\theta,\rho$.

Moreover,  if  $\eta \in pr_{o}B_{r}(go)$ then by Proposition 3.2  $$d(\gamma_{o,\eta}(t),\gamma_{g^{-1}o,\eta}(d(go,o)+t))\leq K$$ for all $t\geq 0$ where $K$ depends only on $r$. Hence, if $\gamma_{o,\eta}[0,t]$ spends a proportion of at least $\rho$ in $Teich_{\epsilon}(S)$ then $\gamma_{o,\eta}[d(o,go),d(o,go)+t]$ spends at a proportion of at least $\rho$ in $Teich_{\epsilon'}(S)$.
Let $E(r,M,\epsilon)$ be the set of $\eta\in PMF$ such that  $\gamma_{o,\eta}[0,t]$ spends a proportion of at least $\rho$ in $Teich_{\epsilon}(S)$ for all $t>M$. Note for large enough $M$ we have $\nu_{o}(E(r,M,\epsilon))>1-\frac{c}{2}$ so 
$$\nu(g^{-1}A(o,go,r,M,\epsilon'))\geq \nu(E(r,M,\epsilon)\cap pr_{g^{-1}o}B_{r}(o))\geq \frac{c}{2}$$ completing the proof.
\end{proof}

Again, by Mumford's compactness criterion and Proposition 3.2 we obtain the following corollary.
\begin{cor}
For each $\theta>0$, $\rho>0$, $\epsilon>0$ with $\mu(Teich_{\epsilon}(S)/Mod(S))\leq \rho$, $R>0$  and $r>0$ larger than twice the diameter of a fundamental domain of $Teich_{\epsilon}(S)$ there is a $K>0$  such that for every $\epsilon'>0$ with $Nbd_{K}Teich_{\epsilon}(S)\subset Teich_{\epsilon'}(S)$ there are $M>0$ and $C>0$  such that for every $x\in Teich_{\epsilon}(S)$ such that  any initial length $\geq R$ segment of $[x,o]$ spends a proportion at least $\theta$ in $Teich_{\epsilon}(S)$ the set of $\eta \in pr_{o}B_{r}(x)$ such that $\gamma_{o,\eta}[d(o,x),d(o,x)+t]$ spends at a proportion of at least $\rho$ in $Teich_{\epsilon}(S)$ for every $t>M$ has measure at least $Ce^{-hd(o,x)}$.
\end{cor}

From now on, we will be able to restrict our attention to $m,n$ such that $m>L_{2},K_{2},\delta_{2}$ and  $n>1000m$ and we will do so without further notice.
\begin{prop}[Proposition 4.2]
For each $b\in \Omega(m,n)$
$$e^{hn-100hm}\leq |Y_{n}(b)|\leq e^{hn+100hm}$$
\end{prop}
\begin{proof}

Assume without loss of generality that $b\in \Omega(m,n)\setminus \bigcup_{k<m} \Omega(n,k)$
Note, for each $x\in Teich(S)$ as $T\to \infty$ the ball

 $B_{T}(\gamma_{x,b}(T))$ converges to $H(x,b,(-\infty,0])$ 

Furthermore, if $$q\in B_{n-35m}(\gamma_{o,b}(n+15m))$$ then by the triangle inequality, $$q\in B_{2n-20}(o)\cap B_{T}(\gamma_{o,b}(T+50m)).$$
On the other hand, suppose $$q\in B_{2n-20m}(o)\cap B_{T}(\gamma_{o,b}(T+50m).$$ 

Since $\gamma_{o,b}[n+15m,n+16m]$ spends at least half the time in $Teich_{\epsilon}(S)$, it follows that $\gamma_{o,b}(n+15m)$ is within $m$ of either $[o,q]$ or $[\gamma_{o,b}(T+50m),q]$. 

Thus, $$ B_{n-35m}(\gamma_{o,b}(n+15m)) \subset B_{2n-20}(o)\cap B_{T}(\gamma_{o,b}(T+50m))\subset  B_{n-33m}(\gamma_{o,b}(n+15m))$$ 
and letting $T\to\infty$ we get $$B_{n-35m}(\gamma_{o,b}(n+15m)) \subset B_{2n-20}(o) \cap H(o,b,(-\infty,-50m]\subset B_{n-33m}(\gamma_{o,b}(n+15m)).$$
Similarly we have:
$$B_{n-35m}(\gamma_{o,b}(n+15m)) \subset B_{2n-20}(o) \cap H(o,b,(-\infty,-50m])\subset B_{n-33m}(\gamma_{o,b}(n+15m))$$
\\
$$B_{n-15m}(\gamma_{o,b}(n+35m)) \subset B_{2n+20}(o) \cap H(o,b,(-\infty,-50m])\subset B_{n-13m}(\gamma_{o,b}(n+35m))$$
\\
$$B_{n-60m}(\gamma_{o,b}(n+40m)) \subset B_{2n-20}(o) \cap H(o,b,(-\infty,-100m])\subset B_{n-58m}(\gamma_{o,b}(n+40m))$$
\\
$$B_{n-40m}(\gamma_{o,b}(n+60m)) \subset B_{2n+20}(o) \cap H(o,b,(-\infty,-100m])\subset B_{n-38m}(\gamma_{o,b}(n+60m))$$

 Since $\gamma_{o,b}(n+35m)$ is within $m$ of $Teich_{\epsilon}(S)$ it is within $m+\rho\leq 2m$ of some point of $Mod(S)o$. Thus,

$$Y_{n}(b)\subset Mod(S)o\cap  B_{2n+20}(o) \cap H(o,b,(-\infty,-50m]$$ $$\subset Mod(S)o\cap B_{n-13m}(\gamma_{o,b}(n+35m))\subset Mod(S)o\cap B_{n-11m}(g_{1}o)$$

 By  the orbit growth estimate of Theorem 1.1  in \cite{ABEM} this implies that $$|Y_{n}(b)|\leq Ce^{h(n-11m)}$$ where $C$ is a uniform constant. This proves the upper bound. 
Now we consider the lower bound; let $W_{n}(b)$ the set of $g\in Mod(S)$ with $d(o,go)\in (2n-20m,2n+20m)$ and $\beta_{b}(go,o)\in[-100m,-50m]$. 
(So $W_{n}(b)$ is the same as $Y_{n}(b)$ but without the thickness assumptions). 
Note, $W_{n}(b)$ is contained in the intersection of $Mod(S)o$ with

$$ (B_{2n+20}(o) \cap H(o,b,(-\infty,-50m])\setminus [(B_{2n+20}(o) \cap H(o,b,(-\infty,-100m]))\cup (B_{2n-20}(o) \cap H(o,b,(-\infty,-50m]))]$$
$$\supset B_{n-15m}(\gamma_{o,b}(n+35m)) \setminus [B_{n-38m}(\gamma_{o,b}(n+60m))\cup B_{n-33m}(\gamma_{o,b}(n+15m))]$$
$$\supset B_{n-16m}(g_{2}o)\setminus (B_{n-37m}(g_{3}o)\cup B_{n-32m}(g_{4}o))$$ for some $g_{2}, g_{3}, g_{4}\in Mod(S)$
By Theorem 1.1 in \cite{ABEM} this implies that $$|W_{n}(b)|\geq Ce^{h(n-16m)}-Ce^{h(n-37m)}-Ce^{h(n-32m)}\geq De^{h(n-16m)}$$
for a uniform constant $D$.
We claim that if $g\in W_{n}(b)$ is such that $\gamma_{g_{2}o,go}[o,t]$ spends a proportion of at least $0.9999$ in $Teich_{\epsilon}(S)$ for all $t>m$ then $g\in W_{n}(b)$. By Theorem 2.10 of \cite{ABEM} at least half of $W_{n}(b)$ satisfy the property, so the proposition follows if the claim is true. Now, we prove the claim. 
Note, $d(g_{2}o,\gamma_{o,b}(n+35m))\leq m$.
\\
Note, $d(go,\gamma_{o,b}(n+40m))\geq n-60m$ so $d(go,\gamma_{o,b}(n+35m))\geq n-65m$ and hence $d(go,g_{2}o)\geq n-66m$.
\\
So, we have $$n-66m\leq d(go,g_{2}o)\leq n-12m$$  $$n+34 m\leq d(o,g_{2}o)\leq n+36m$$ and $$2n-20m\leq d(o,go)\leq 2n+20m.$$ 

If $\gamma_{g_{2}o,go}(t)$ is within a $\delta$ neighborhood of $[o,g_{2}o]$ we must therefore have $$2n-20m\leq d(o,go)\leq d(go,g_{2}o)+d(o,g_{2}o)+\delta-t \leq 2n+24m+\delta-t$$ so $$t\leq 44m+\delta\leq 45m$$. 
Hence if  $\gamma_{go,g_{2}o}(t)$ is within a $\delta$ neighborhood of $[o,g_{2}o]$ we must have $t>d(go,g_{2}o)-45m>n-120m$. Note, $\gamma_{go,g_{2}o}[n-121m,n-120m]$ spends at least $90$ percent in $Teich_{\epsilon}(S)$ so there is a $t\in [n-121m,n-120m]$
with $\gamma_{go,g_{2}o}(t)\in Teich_{\epsilon}(S)$ and $\gamma_{go,g_{2}o}(t)$ within $\delta$ of $[o,go]\cup [o,g_{2}o]$. Since $t<n-120m$, $\gamma_{go,g_{2}o}(t)$ cannot be within $\delta$ of $[o,g_{2}o]$ so there is an $s\in [-\delta,\delta]$ with 
$d(\gamma_{go,g_{2}o}(t),\gamma_{go,o}(t+s))\leq \delta$.  By Proposition 3.2 we have $d(\gamma_{go,g_{2}o}(t),\gamma_{go,o}(t+s))\leq K$ for all $t<n-121m$ so at least 99 percent of $\gamma_{go,o}[n-121m-t,n-121m]$ lies in $Teich_{\epsilon'}(S)$ for all $n-121m\geq t\geq m$.
\end{proof}

\begin{prop}
If $b\in \Omega(n,m)$ and $g\in Y_{n}(b)$ then for some $t\in [n-9m,n-8m]$ with $b(t)\in Teich_{\epsilon}(S)$  we have 
$d(b(t),\gamma_{o,go}(t))\leq 2 \delta$
\end{prop}
\begin{proof}

Assume without loss of generality that $b\in \Omega(n,m)\setminus \bigcup_{k<m} \Omega(n,k)$
Note, $$go\in B_{2n+20}(o)\cap H(o,b,(-\infty,-50])\subset B_{n-13m}(b(n+35m)).$$ 
Since $\gamma_{o,b}[n-10m,n-9m]$ spends more than 0.9 of the time in $Teich_{\epsilon}(S)$ there is a $t\in [n-10m,n-9m]$ with $b(t)\in Teich_{\epsilon}(S)$ such that $b(t)$ is within $\delta$ of $$[o,go]\cup[b(n),go].$$ However, any point of $[b(n+35m),go]$ is within $n-13m$ of $go$ while 
$$d(b(t),go)\geq d(o,go)-d(b(t),o)\geq (2n-20m)-(n-9m)=n-11m> n-13m+\delta$$. Thus, $b(t)$ is within $\delta$ of $[o,go]$ completing the proof.
\end{proof}

We therefore obtain:
\begin{cor}
If $b\in \Omega(n,m)$,  $g\in Y_{n}(b)$, then for all $t\leq n-9m$ we have 
$d(b(t),\gamma_{o,go}(t))\leq K$
\end{cor}

\begin{cor}
If $b\in \Omega(n,m)$,  $g\in Y_{n}(b)$, $b'\in Z_{n}(g)$  then for all $t\leq n-20m$ we have 
$d(b'(t),b(t))\leq K+K_{1}\leq 2K_{1}$
\end{cor}

Using this and the shadow estimate from Lemma 5.1 we obtain
\begin{cor}[Proposition 4.4]
For almost every $b'\in PMF$
$$\nu\{b\in \Omega(n,m):b'\in \bigcup_{g\in Y_{n}(b)}Z_{n}(g)\} \leq Ce^{10hm-hn}$$ where $C$ does not depend on $m,n$.
\end{cor}

\begin{prop}
If $b \in \Omega(n,m)$ and $g\in Y_{n}(b)$ then for every $t>n+61m$ we have $d(b(t),\gamma_{o,go}(t))\geq 2K$
\end{prop}
\begin{proof}
Note, for large enough $T$ we have $$d(go,b(T))-T=d(go,b(T))-d(o,b(T))\geq -101m$$
If $d(b(t),\gamma_{o,go}(t))\leq 2K$ then $$T-101m\leq d(go,b(T))\leq d(b(T),b(t))+2K+d(go,\gamma_{go,o}(t))\leq (T-t)+2K+(2n+20m-t)$$ so 
$$2t\leq 2n+121m+2K\leq 2n+121m$$ so $t\leq n+61m$.
\end{proof}

\begin{cor}
If $b \in \Omega(n,m)$ and $g\in Y_{n}(b)$then there exists an $s\in [-\delta,\delta]$ such that $d(b(t),\gamma_{go,b}(t+s))\leq K$ for all $t>62m$
\end{cor}

\begin{proof}
Since $\gamma_{o,b}[n+61m,n+62m]$ spends at a proportion of at least 0.9 in $Teich_{\epsilon}(S)$ there is  a $t_{0}\in [n+61m,n+62m]$ with $b(t_{0})\in Teich_{\epsilon}(S)$ such that 
 $d(b(t_{0}),\gamma_{go,b}(t_{0}+s))\leq \delta$ for some $|s|<\delta$. By Proposition 3.2,  $d(b(t),\gamma_{go,b}(t+s))\leq K$ for all $t>t_{0}$.
\end{proof}

\begin{prop}[Proposition 4.6]
For each uniquely ergodic $b\in PMF$
$$|\{g\in Mod(S):gb\in \Omega(n,m),g\in Y_{n}(gb)\}|\lesssim_{m}e^{hn}$$
\end{prop}

\begin{proof}
If $g\in Y_{n}(gb)$ then 
$$d(g^{-1}o,o)\leq 2n+20m$$ and
$$\beta_{b}(g^{-1}o,o)=-\beta_{b}(o,g^{-1}o)= -\beta_{gb}(go,o)\leq 100m$$

Moreover, if $gb\in \Omega(n,m)$ then there exists an $s\in [-\delta,\delta]$ such that $$d(\gamma_{o,gb}(t),\gamma_{go,gb}(t+s))\leq K$$ for all $t>62m$. Hence, at least $60$ percent of $\gamma_{go,gb}[n+65m,n+66m]$ lies in $Teich_{\epsilon'}(S)$.
Note, as $T\to \infty$ we have $B(\gamma_{o,b}(T-100m))\to H(o,b,(\-infty, 100m])$. Suppose $q\in B_{2n+20m}(o)\cap B(\gamma_{o,b}(T-100m))$. 
Since $\gamma_{o,b}[n+61m,n+62m]$ spends at least $60$ percent in $Teich_{\epsilon'}(S)$, it follows that $\gamma_{o,b}(n-40m)$ is within $105m$ of either $[o,q]$ or $[\gamma_{o,b}(T-100m),q]$. 
In the first case,
$$d(\gamma_{o,b}(n-40m),q)\leq d(o,q) - d(\gamma_{o,b}(n-40m),o)+50m\leq (2n+20m)- (n-40m)+210m\leq n+270m$$
Similarly, in the second case
$$d(\gamma_{o,b}(n-40m),q)\leq n+270m$$
So, letting $T\to \infty$ we get $$B_{2n+20m}(o)\cap H(o,b,(\-infty, 100m])\subset B_{n+270m}\gamma_{o,b}(n-40m)\subset B_{n+400m}(g_{5}o)$$ for some $g_{5}\in Mod(S)$.
Thus by Theorem 1.1 of \cite{ABEM} $$|Mod(S)o\cap B_{2n+20m}(o)\cap H(o,b,(-\infty, 100m])|\leq Ce^{hn+400hm}$$ for some uniform constant $C$. 
\end{proof}

\begin{prop}[Proposition 4.5]
For all $b'\in PMF$
the number of $g\in Mod(S)$ with $gb'\in Z_{n}(g)$ and $g\in Y_{n}(b)$ for some $b\in \Omega(n,m)$ has cardinality
$\lesssim_{m}e^{hn}$.
\end{prop}
\begin{proof}
If $gb'\in Z_{n}(g)$ then $$\beta_{b'}(g^{-1}o,o)=-\beta_{gb'}(go,o)\leq 100m$$ and $d(g^{-1}o,o)\leq 2n+20m$.
Moreover, at least $60$ percent of $\gamma_{o,b'}[n-10m,n-9m]$ lies in $Teich_{\epsilon'}(S)$ 
 so the result follows by the same argument as Proposition 4.6.
\end{proof} 

\begin{prop}[Proposition 4.3]
For each $b\in \Omega(n,m)$ and $g\in Y_{n}(b)$  we have 
$$Ce^{9hm-hn}\leq \nu(Z_{n}(g))\leq De^{21hm-hn}$$ for some uniform constants $C$ and $D$
\end{prop}

\begin{proof}
Assume without loss of generality that $b\in \Omega(n,m)\setminus \bigcup_{k<m} \Omega(n,k)$.

Let $t \in [n-21m,n-20m]$ with $b(t)\in Teich_{\epsilon}(S)$.
We have $d(b(t),b'(t))\leq 2K_{1}$ and hence $b'\in pr_{o}B_{2K_{1}}(b(t))$.
By definition of $\Omega(n,m)$ for all $s>2m$ the segment $b([t-s,t])$ spends at least half the time in $Teich_{\epsilon}(S)$.
By Proposition 3.1 this implies $$\nu_{o}(Z_{n}(g))\leq \nu_{o}(pr_{o}B_{2K}(b(t)))\leq Ce^{21hm-hn}$$ where $C$ is independent of $m,n$.
For the lower bound, consider $t_{1}\in [n-20m,n-19m]$, $t_{2}\in [n-10m,n-9m]$ with $b(t_{i})\in Teich_{\epsilon}(S)$. Note,  we have $d(b(t_{i}),\gamma_{o,go}(t_{i}))\leq K$ so $\gamma_{o,go}(t_{i})\in Teich_{\epsilon'}(S)$.
Moreover, $b([t_{i}-t,t_{i}])$ spends at least half the time in $Teich_{\epsilon}(S)$ for each $t\in [2m,t_{i}]$ so $\gamma_{o,go}([t_{i}-t,t_{i}])$ spends at least half the time in $Teich_{\epsilon'}(S)$ for each $t\in [2m,t_{i}]$.
Note, if $d(b'(t_{1}),\gamma_{o,go}(t_{1}))\leq 2\delta_{1}$ then $d(b'(t),\gamma_{o,go}(t))\leq K_{1}$ for all $t\leq n-20m$.
Thus, $Z_{n}(g)$ contains all the $$b'\in pr_{o}B_{2\delta_{1}}(\gamma_{o,go}(t_{1})\setminus pr_{o}B_{K_{1}}(\gamma_{o,go}(t_{1})$$ such that $b'([n-9m,n-8m])$ spends more than $90$ percent of the time in  $Teich_{\epsilon''}(S)$.
By Proposition 4.4, the $\nu$ measure of the  $b'\in pr_{o}B_{2\delta_{1}}(\gamma_{o,go}(t_{1}))$ such that $b'([n-9m,n-8m])$ spends more than $90$ percent of the time in  $Teich_{\epsilon''}(S)$ is  $\geq Ce^{21hm-hn}$ and $\nu_{o}( pr_{o}B_{K_{1}}(\gamma_{o,go}(t_{1}0)))\geq De^{9hm-hn}$ for $C,D$ independent of $n,m$. Thus, for $e^{12hm}>2D/C$ we have $$\nu(Z_{n}(g))\geq \frac{D}{2}e^{9hm-hn}$$ and so obtain the desired result.

\end{proof}

\begin{prop}[Proposition 4.8]
For each $b\in \Omega(n,m)$, $g\in Y_{n}(b)$ and $b'\in Z_{n}(g)$ we have
$-6m\leq \beta_{b'}(go,o)\leq 21m$
\end{prop}

\begin{proof}
Assume without loss of generality that $b\notin \Omega(n,k)$ for any $k<m$.
Assume $T>n>1000m$.
Note, $d(b'(n-20m),\gamma_{o,go}(n-20m))\leq K_{1}$ so $$d(go,b'(T))\leq d(b'(n-20m),\gamma_{o,go}(n-20m))+d(b'(n-20m),b(T))+d(\gamma_{o,go}(n-20m),go)$$ $$\leq K_{1}+(T-n+20m)+[(2n-20m)-(n-20m)]=K_{1}+T+20m$$
So $$d(b'(T),go)-d(b'(T),o)\leq K_{1}+20m\leq 21m$$ for all $T$ so $$\beta_{b'}(go,o)\leq 21m$$
On the other hand, for some $t\in [n-10m,n-9m]$ we have $d(\gamma_{o,go}(t),b'(t))\geq K_{1}$ and so for all $t\geq n-9m$ we have $d(\gamma_{o,go}(t),b'(t))\geq 2\delta_{1}$. Since at least $60$ percent of $\gamma_{o,go}[n-9m,n-8m]$ and $b'([n-9m,n-8m])$ lies in $Teich_{\epsilon''}(S)$, it follows that $\gamma_{o,go}(n-9m)$ and $b'(n-9m)$ are both within $m/2+\delta_{2}<m$ of points  on $[go,b'(T)]$. 
Thus, $$d(go,b'(T))+4m\geq d(go,\gamma_{o,go}(n-9m))+d(b'(n-9m),\gamma_{o,go}(n-9m))+d(b'(n-9m),b(T))$$ $$\geq [(2n-20m)-(n-9m)]+(T-n+9m)=T-2m.$$
Hence, $$d(go,b'(T))-d(o,b'(T))\geq -6m$$ and letting $T\to \infty$ we get $$\beta_{b'}(go,o)\geq -6m$$
\end{proof}

\begin{prop}
If $b \in \Omega(n,m)$ and $g\in Y_{n}(b)$ then for every $t>n+61m$ we have $d(b'(t),\gamma_{o,go}(t))\geq 2K_{1}$
\end{prop}
\begin{proof}
This is proved in the same way as Proposition 4.10.
\end{proof}

\begin{prop}[Proposition 4.7]
For each $b\in \Omega(n,m)$, $g\in Y_{n}(b)$, $b'\in Z_{n}(g)$ we have 
$$b,b',g^{-1}b,g^{-1}b'\in pr_{o,g^{-1}o}B_{\delta_{1}}(\gamma_{o,g^{-1}o}(t))$$
 for some $t>n-122m$ with $\gamma_{o,g^{-1}o}(t)\in Teich_{\epsilon'}(S)$
\end{prop}

\begin{proof}
It is enough to prove that 
$$b,b'\in pr_{go,o}B_{\delta_{1}}(\gamma_{go,o}(t))$$

Note, $\gamma_{go,o}([n-122m,n-121m])$ spends at least $90$ percent in $Teich_{\epsilon'}(S)$, so there is a $t\in [n-122m,n-121m]$ with $\gamma_{go,o}(t)\in Teich_{\epsilon'}(S)$ so that $\gamma_{go,o}(t)$ is within $\delta_{1}$ of $[go,b)\cup [o,b)$ and also of $[go,b')\cup [o,b')$ Note, $\gamma_{go,o}(t)=\gamma_{o,go}(s)$ for $s=d(go,o)-t\geq 2n-20m-(n-121m)>n+100$ so we must have points $\gamma_{go,b}(t_{1})$ and $\gamma_{go,b'}(t_{2})$ within $\delta_{1}$ of $\gamma_{go,o}(t)$. By Proposition 3.2 we obtain the desired result.
\end{proof}

\end{document}